\documentclass[12pt]{elsarticle}

\usepackage{amssymb,amsmath,amsfonts,mathrsfs}
\usepackage{amsthm}
\usepackage[pdftex,colorlinks]{hyperref}

\newtheorem{theorem}{Theorem}[section]
\newtheorem{remark}{Remark}[section]
\newtheorem{corollary}[theorem]{Corollary}

\newtheorem{lemma}[theorem]{Lemma}

\newtheorem{definition}{Definition}[section]

\newcommand{\cH}{\mathcal{H}}

\newcommand{\cD}{\mathcal{D}}
\newcommand{\cC}{\mathcal{C}}
\newcommand{\cM}{\mathcal{M}}
\newcommand{\cV}{\mathcal{V}}
\newcommand{\cW}{\mathcal{W}}

\newcommand{\bbR}{\mathbb{R}}
\newcommand{\bbC}{\mathbb{C}}
\newcommand{\bbZ}{\mathbb{Z}}
\newcommand{\bbE}{\mathbb{E}}
\newcommand{\bbD}{\mathbb{D}}
\newcommand{\bbM}{\mathbb{M}}
\newcommand{\bbI}{\mathbb{I}}
\newcommand{\bbP}{\mathbb{P}}

\newcommand{\sB}{\mathscr{B}}

\renewcommand{\Re}[1]{\mathrm{Re}\{#1\}}
\renewcommand{\Im}[1]{\mathrm{Im}\{#1\}}

\newcommand{\iprod}[1]{\left\langle #1 \right\rangle}
\newcommand{\bra}[1]{\left\langle#1\right|}
\newcommand{\ket}[1]{\left|#1\right\rangle}
\newcommand{\kebr}[1]{\left|#1\right\rangle\left\langle #1 \right|}
\newcommand{\wto}{\hookrightarrow}
\DeclareMathOperator{\supp}{\mathrm{supp}}
\newcommand{\tr}[1]{\mathrm{Tr}\left\{#1 \right\}}

\DeclareMathOperator*{\argmin}{argmin}

\begin{document}

\begin{frontmatter}

\title{Semi-Global Approximate stabilization of an infinite dimensional quantum stochastic system}

 \author[label1,label2]{Ram Somaraju}
 \author[label1,label2]{Mazyar Mirrahimi}
\address[label1]{INRIA Rocquencourt, Domaine
de Voluceau, B.P. 105, 78153 Le Chesnay cedex, France,
(ram.somaraju, mazyar.mirrahimi)@inria.fr}
\address[label2]{Ram Somaraju and Mazyar Mirrahimi acknowledge support from ``Agence Nationale de la Recherche'' (ANR), Projet Jeunes Chercheurs EPOQ2 number ANR-09-JCJC-0070.} 
\author[label3]{Pierre Rouchon}
\address[label3]{P. Rouchon is with Mines ParisTech, Centre Automatique et Syst\'{e}mes, Math\'{e}matiques et Syst\'{e}mes, 60 Bd Saint Michel, 75272 Paris cedex 06,
France, pierre.rouchon@mines-paristech.fr}
\address[label4]{Pierre Rouchon acknowledges support from ANR (CQUID).}

\begin{abstract}
In this paper we study the semi-global (approximate) state feedback stabilization of an infinite dimensional quantum stochastic system towards a target state. A discrete-time Markov chain on an infinite-dimensional Hilbert space is used to model the dynamics of a quantum optical cavity. We can choose an (unbounded) strict Lyapunov function that is minimized at each time-step in order to prove (weak-$\ast$) convergence of probability measures to a final state that is concentrated on the target state with (a pre-specified) probability that may be made arbitrarily close to $1$. The feedback parameters and the Lyapunov function are chosen so that the stochastic flow that describes the Markov process may be shown to be tight (concentrated on a compact set with probability arbitrarily close to $1$). We then use Prohorov's theorem and properties of the Lyapunov function to prove the desired convergence result.
\end{abstract}


\end{frontmatter}


\section{Introduction}
\label{sec:int} 
In this paper we consider the stabilization of a discrete-time Markov process (Equations~\eqref{eqn:rhohalf} and~\eqref{eqn:rhoone}) that is defined on the unit sphere on an infinite-dimensional Hilbert (Fock) space $\cH$ at a target state which is a specific unit vector in $\cH$ corresponding to a photon-number state. We consider a Lyapunov function based state-feedback controller that drives our quantum system to the target state with probability greater than some pre-specified probability $1-\epsilon$ for all $1 > \epsilon > 0$\footnote{The problem of output feedback control has been examined in the finite-dimensional context in~\cite{Mirrahimi2010,Dotsenko2009} using a quantum adaptation of the Kalman filter. We do not discuss the problem of estimating the state of the system and refer the reader to~\cite{Mirrahimi2010} for further details on designing a state estimator.}.

The specific physical system under consideration uses Quantum Non-Demolition (QND) measurements to detect and/or produce highly nonclassical states of light in trapped super-conducting cavities~\cite{Deleglise2008,Gleyzes2007,Guerlin2007} (see~\cite[Ch. 5]{Haroche2006} for a description of such quantum electro-dynamical  systems and~\cite{Brune1992} for detailed physical models with QND measures of light using atoms). In this paper we examine the feedback stabilization of such experimental setups near a pre-specified target photon number state. Such photon number states, with a precisely defined number of photons, are highly non-classical and have potential applications in quantum information and computation.

As the Hilbert space $\cH$ is infinite dimensional it is difficult to design feedback controllers to drive the system towards a target state  (because closed and bounded subsets of $\cH$ are not compact). In~\cite{Mirrahimi2010,Dotsenko2009} a controller was designed by approximating the underlying Hilbert space $\cH$ with a finite-dimensional Galerkin approximation $\cH_{N_{max}}$. Physically this approximation leads to an artificial bound $N_{max}$ on the maximum number of photons that may be inside the cavity. In this paper we wish to design a controller for the full Hilbert space $\cH$ without using the finite dimensional approximation. Simulations (see~\cite{Somaraju2011}) indicate that the controller in Theorem~\ref{the:mainRes} below performs better than the one designed using the finite-dimensional approximation in~\cite{Mirrahimi2010,Dotsenko2009}.

Controlling infinite dimensional quantum systems have previously been examined in the deterministic setting of partial differential equations, which do not involve quantum measurements. Various approaches have been used to overcome the non-compactness of closed and bounded sets. One approach consists of proving approximate convergence results which show convergence to a neighborhood of the target state for example in~\cite{Beauchard2009,Mirrahimi2009}. Alternatively, one examines weak convergence for example, in~\cite{Beauchard2010}. Other approaches such as using strict Lyapunov functions or strong convergence under restrictions on possible trajectories to compact sets have also been used in the context of infinite dimensional state-space for example in~\cite{Coron1998,Coron2007}.

The situation in our paper is different in the sense that the system under consideration is inherently stochastic due to quantum measurements. The system we consider may be described using a discrete time Markov process on the set of unit vectors in the state Hilbert space $\cH$ as explained in Subsection~\ref{sub:DTMC}. We use a strict Lyapunov function that restricts the system trajectories with high probability to compact sets as explained in Section~\ref{sec:mainRes}. We use the properties of weak-convergence of measures to show approximate convergence (i.e. with probability of convergence approaching one) of the discrete time Markov process towards the target state. 
\subsection{Outline}
The remainder of the paper is organised as follows: in the following Section~\ref{sec:defs} we introduce some notation and the system model of the discrete-time Markov process. We also recall some results concerning the (weak-$\ast$)-convergence of probability measures. In Section~\ref{sec:mainRes} we state the main result of our paper (Theorem~\ref{the:mainRes}) concerning the approximate semi-global stabilizability of the Markov process at our target state. We also provide a proof of the main result using several Lemmas. We then present our conclusions in the final Section.
\section{Definitions and System Description}\label{sec:defs}
We introduce some notation that will be used to describe the discrete-time Markov process that characterizes the system.
\subsection{Notation}
In this paper, we use Dirac's Bra-ket notation commonly used in Physics literature\footnote{See e.g.~\cite[Sec. 8.3]{Teschl2009} for more details of the quantum Harmonic oscillator model and notation used here.}. The system Hilbert space associated with the quantum cavity is a Fock space which we denote by $\cH$ with inner-product $\iprod{\cdot|\cdot}_\cH$ and norm $\|\cdot\|_\cH$. We drop the subscript, for ease of notation, if this causes no confusion. Let the set\footnote{$\bbZ$,$\bbZ^+$ and $\bbZ^+_0$ denote the set of integers, positive integers and non-negative integers, respectively.} 
\begin{equation*}
\{\ket{n}: n \in \bbZ^+_0\}
\end{equation*}
denote the canonical basis of the Fock space $\cH$. Physically, the state $\ket{n}$ represents a cavity state with precisely $n$ photons. 

Let $a$ and $a^\dagger$ be the annihilation and creation operators defined on domains $\cD(a) \subset \cH$ and $\cD(a^\dagger)\subset \cH$, respectively and $N = a^\dagger a$ be the number operator with domain $\cD(N)$. These unbounded operators satisfy the relations
\begin{equation}\label{eqn:aadagN}
a\ket{n} = \sqrt{n}\ket{n-1},\ \ \  a^\dagger\ket{n} = \sqrt{n+1}\ket{n+1},\ \ \  N\ket{n} = n\ket{n}
\end{equation}
for all $n\in \bbZ^+_0$. 

For all $x\in \cH$ and $\epsilon > 0$ denote by $B_\epsilon(x)$ the open ball in $\cH$ centered at $x$ and of radius $\epsilon$. Also denote by $\bar{B}_1 = \{x\in \cH:\|x\|=1\}$ the closed set of unit vectors in $\cH$ with the topology inherited from $\cH$. Let $B'_\epsilon(x)$ denote the set $\bar{B}_1\cap B_\epsilon(x)$ for all $x\in \bar{B}_1$ and $\epsilon > 0$. Let $\cC = (\cC(\bar{B}_1),\|\cdot\|_\infty)$ be the Banach space of continuous functions on $\bar{B}_1$ with the supremum norm $\|\cdot\|_\infty$. 

We denote by $\sB = \sB(\bar{B}_1)$ the Borel $\sigma$-algebra of $\bar{B}_1$ and by $\cM_1$ the set of all probability measures on the measure space $(\bar{B}_1,\sB)$. For all $\mu\in \cM_1$ and $\sB$-measurable functions $f$ defined on $\bar{B}_1$ we denote by 
\begin{equation*}
\bbE_\mu[f] = \int_{x\in \bar{B}_1} f(x) d\mu(x)
\end{equation*}
the expectation value of the function $f$ with respect to measure $\mu$. 

The \emph{support} of a probability measure $\mu$ is defined to be the set
\begin{equation*}
\supp(\mu) = \{x\in \bar{B}_1: \mu(V) > 0 \textrm{ for all open neighborhoods $V$ of $x$}\}.
\end{equation*}
\subsection{Topology on $\cM_1$}
In this paper we study the weak-$\ast$ convergence of probability measures. It can be shown (see e.g.~\cite{Merkle2000}) that the space $\cM_1$ is a subset of the unit ball in the (continuous linear) dual space $\cC^\ast$ of $\cC$ through the relation $\mu\in \cM_1 \mapsto \Lambda_\mu\in \cC^\ast$ where
\begin{equation*}
\Lambda_\mu(f) = \bbE_\mu [f]
\end{equation*} 
for all $f\in \cC$. When we refer to the convergence of a sequence of measures, we mean converges with respect to the weak-$\ast$ topology of $\cC^\ast$.

\begin{definition}\label{defn:convMeas} We say that a sequence of probability measure $\{\mu_n\}_{n=1}^\infty\subset \cM_1$ converges (weak-$\ast$) to a probability measure $\mu\in \cM_1$ if for all $f\in \cC$
\begin{equation*}
\lim_{n\to\infty} \bbE_{\mu_n}[f] = \bbE_\mu[f]
\end{equation*}
and we write
\begin{equation*}
\mu_n\wto \mu.
\end{equation*}
\end{definition}

In the weak-$\ast$ topology on the set of probability measures, compactness is related to the notion of tightness of measures. A set of probability measures $S\subset \cM_1$ is said to be \emph{tight}~\cite[p. 9]{Billingsley1999} if for all $\epsilon > 0$ there exists a compact set $K_\epsilon \subset \bar{B}_1$ such that for all $\mu\in S$,
\begin{equation*}
\mu[\bar{B}_1\setminus K_\epsilon] < \epsilon.
\end{equation*}
We recall below Prohorov's theorem (see e.g.~\cite{Merkle2000}).
\begin{theorem}[Prohorov's theorem]\label{the:prohorov}
Any tight sequence of probability measures has a (weak-$\ast$) converging subsequence.
\end{theorem}

\subsection{Discrete-time Markov process}\label{sub:DTMC}
We now describe the evolution of our quantum system which is governed by a Markov process on the state space $\bar{B}_1$. We introduce below the Markov process model with minimal references to the actual physical system under consideration. We refer the interested reader to~\cite[Ch. 5]{Haroche2006} and references therein for a description of the physical system and the approximations involved in deriving the Markov process model (also see~\cite{Mirrahimi2010,Dotsenko2009}). 

Define the displacement operator $D_\alpha:\cH\to\cH$ and measurement operators $M_s:\cH\to \cH$ where $\alpha\in \bbR$ and $s\in \{g,e\}$ as
\begin{equation*}
D_\alpha = \exp\{\alpha (a^\dagger - a)\},\ M_g = \cos(\theta + N\phi)\textrm{ and } M_e = \sin(\theta + N\phi).
\end{equation*}
Here, $\theta$ and $\phi$ are experimentally determined real numbers and the operators $a$, $a^\dagger$ and $N$ are defined in Equation~\eqref{eqn:aadagN}. Recall that because the operator $i(a^\dagger - a)$ is self-adjoint,  we may conclude from Stone's theorem that the set of operators $\{D_\alpha:\alpha\in \bbR\}$ form a strongly-continuous unitary group (see e.g.~\cite[Sec V111.4]{Reed1980}), i.e.
\begin{equation}\label{eqn:stCont}
\lim_{\alpha\to\alpha_0} \|D_\alpha\ket{\psi} - D_{\alpha_0}\ket{\psi}\|_\cH = 0,\ \forall \psi\in \cH.
\end{equation}

Denote by $\ket{\psi_k}\in \bar{B}_1$ the state of the system at time-step $k$. Given the state $\ket{\psi_k}$ at time-step $k$ the state $\ket{\psi_{k+1}}$ at time-step $k+1$ is a random variable whose distribution is given by the following two-step equation
\begin{eqnarray}
\ket{\psi_{k+1/2}} &=&\frac{ M_s \ket{\psi_k}}{\|M_s\ket{\psi_k}\|_\cH} \textrm{ with probability } P_s = \|M_s\ket{\psi_k}\|_\cH^2,\label{eqn:rhohalf}\\
\ket{\psi_{k+1}} &=& D_{\alpha_k} \ket{\psi_{k+1/2}}\label{eqn:rhoone}.
\end{eqnarray}
Here $s\in \{e,g\}$ and the control $\alpha_k\in \bbR$.
\begin{remark}
The time evolution from the step $k$ to $k + 1$, consists of two types of evolutions: a projective measurement by the operators  $M_s$ and a control part involving operator $D_\alpha$. For the sake of simplicity, we will use the notation of $\ket{\psi_{k+1/2}}$ to illustrate this intermediate step.
\end{remark}

The Markov jump probabilities may also be written in terms of density operators. Given any $\ket{\psi}$ denote by $\rho_\psi$ the density operator $\ket{\psi}\bra{\psi}$. Then
\begin{equation*}
|\iprod{\psi|n}|^2 = \tr{\rho_\psi\kebr{n}}\textrm{ and }  \|M_g\ket{\psi}\|^2 = \tr{M_g^2\rho_\psi}.
\end{equation*}
Here $\tr{\cdot}$ is the trace of a trace-class operator on $\cH$. If we set $\rho_k  = \kebr{\psi_k}$ then we can write the Markov jump probabilities~\eqref{eqn:rhohalf}, \eqref{eqn:rhoone} using the equivalent density operator description.
\begin{eqnarray*}
\rho_{k+1/2} &=&\frac{\bbM_s(\rho_k)}{\tr{\bbM_s(\rho_k)}} \textrm{ with probability } P_s =\tr{\bbM_s(\rho_k)}\\
\rho_k &=&\bbD_{\alpha_k}(\rho_k).
\end{eqnarray*}
Here, $\bbD_\alpha(\cdot) \triangleq D_{\alpha}\cdot D_{-\alpha}$ and $\bbM_s(\cdot) = M_s\cdot M_s$ are super-operators. We switch between the two equivalent descriptions throughout the paper depending on convenience. 

\begin{remark}
Equations~\eqref{eqn:rhohalf} and~\eqref{eqn:rhoone} determine a stochastic flow in $\cM_1$ and we denote by $\Gamma_k(\mu_0)$ the probability distribution of $\ket{\psi_k}$, given $\mu_0$, the probability distribution of $\ket{\psi_0}$. 
\end{remark}
\subsection{Some useful formulas}
We recall below some useful results. The Baker-Campbell-Hausdorff formula, which will be used to evaluate the derivatives of our Lyapunov function, states (see e.g.~\cite[p. 291]{Nielsen2000})
\begin{equation}\label{eqn:bch}
\exp(\alpha H) A \exp(-\alpha H) = \sum_{n=0}^\infty \frac{\alpha^n}{n!} C_n,
\end{equation}
where $A,H$ and $C_n$ are linear operators on $\cH$ and $\alpha\in \bbC$. The $C_n$ are defined recursively with $C_0 = A$ and $C_{n+1} = [H,C_n]$ for $n\geq 0$. 

Let $X_n$ be a Markov process on some state space X. Suppose that there is a non-negative function $V$ on $X$ satisfying $\bbE[V(X_1)|X_0 = x)]- V (x) \leq 0$, then Doob's inequality states
\begin{equation}\label{eqn:doob}
\bbP\left(\sup_{n\geq 0}V(X_n) \geq \gamma|X_0 = x\right) \leq \frac{V(x)}{\gamma}.
\end{equation}
\section{Main Results}\label{sec:mainRes}
We prove the main results of our paper in this Section. We wish to use the control $\alpha_k$ to drive the system into a pre-specified target state $\ket{\bar{n}}$ where $\bar{n}\in \bbZ^+_0$. 

We use a strict Lyapunov function $V:\bar{B}_1\to [0,\infty]$ defined\footnote{We choose this Lyapunov function assuming $\bar{n} \geq 2$. The Lyapunov function may easily be modified for the case $\bar{n} = 0,1$ and all the proofs in this paper may be applied to that case as well.}
\begin{equation}\label{eqn:lyap}
V(\ket{\psi}) = \sum_{n=0}^\infty \sigma_n\left|\iprod{\psi|n}\right|^2 + \delta\left(\cos^4(\phi_{\bar{n}}) + \sin^4(\phi_{\bar{n}})   - \left\|M_g\ket{\psi} \right\|^4 - \left\|M_e \ket{\psi}\right\|^4\right).
\end{equation}
Here 
\begin{equation*}
\phi_{n} = \theta + n\phi = \cos^{-1}\left(\|M_g\ket{n}\|\right) = \sin^{-1}\left(\|M_e\ket{n}\|\right),\ n = 0,1,2,\ldots.
\end{equation*}
$\delta > 0$ is a small positive number and
\begin{equation}\label{eqn:sigmaDef}
\sigma_n = \left\{\begin{array}{ll} 
\frac{1}{8} +\sum_{k=1}^{\bar{n}} \frac{1}{k} - \frac{1}{k^2},& \textrm{ if $n=0$}\\
\sum_{k=n+1}^{\bar{n}} \frac{1}{k} - \frac{1}{k^2},& \textrm{ if $1 \leq n < \bar{n}$}\\
0, &  \textrm{ if $n = \bar{n}$ }\\
\sum_{k=\bar{n}+1}^n \frac{1}{k} + \frac{1}{k^2},& \textrm{ if $n > \bar{n}$}
\end{array}\right.
\end{equation}
We set $\cD(V)\subset \bar{B}_1$ to be the set of all $\ket{\psi}\in \bar{B}_1$ where the above Lyapunov function  is finite. 
\begin{remark}
We note that coherent states $\ket{\xi_\alpha} = \exp(-|\alpha|^2/2)\sum_{n = 0}^\infty \frac{\alpha^n}{n!}\ket{n}$, for $\alpha\in \bbC$ are in $\cD(V)$. These coherent states naturally occur in optical cavities and are generally the initial condition $\ket{\psi_0}$ is a coherent state in experiments.
\end{remark}
We choose a feedback that maximises the expectation value of the Lyapunov function in every time-step $k$ - 
\begin{equation}\label{eqn:control}
\alpha_k = \argmin_{\alpha\in [-\bar{\alpha},\bar{\alpha}]} V\left(D_\alpha\ket{\psi_{k+1/2}}\right)
\end{equation}
for some positive constant $\bar{\alpha}$. 

\begin{remark} The Lyapunov function and feedback $\alpha_k$ are chosen  to be this specific form to serve three purposes - 
\begin{enumerate}
\item We choose the sequence $\sigma_n \to \infty$ as $n\to \infty$. This guarantees that if we choose $\alpha_k$ to minimize the  expectation value of the Lyapunov function then the trajectories of the Markov process are restricted to a compact set in $\bar{B}_1$ with probability arbitrarily close to 1. This implies that the $\omega$-limit set of the process is non-empty  (see Lemma~\ref{lem:convSub}).
\item The term $-\delta(\|M_g\ket{\psi}\|^4 +\|M_g\ket{\psi}\|^4)$ is chosen such that the Lyapunov function is a strict Lyapunov functions for the Fock states. This implies that the support of the $\omega$-limit set only contains Fock states (see Lemma~\ref{lem:supp}).
\item The relative magnitudes of the coefficients $\sigma_n$ have been chosen such that $V(\ket{\bar{n}})$ is a strict global minimum of $V$. Moreover given any $M > \bar{n}$ we can choose $\delta,\bar{\alpha}$ such that for all $M \geq m\neq \bar{n}$, and for all $\ket{\psi}$ in some neighborhood of $\ket{m}$, $V(D_\alpha(\ket{\psi})$ does not have a local minimum at $\alpha = 0$. This implies that if $\ket{\psi_k}$ is in this neighborhood of $\ket{m}$ then we can choose an $\alpha_k\in [-\bar{\alpha},\bar{\alpha}]$ to decrease the Lyapunov function and move $\ket{\psi_k}$ away from $\ket{m}$ by some finite distance with probability that can be made arbitrarily close to 1 by an appropriate choice of $\bar{\alpha}$ and $\delta$ (see proof of Lemma~\ref{lem:probFock}).
\end{enumerate}
\end{remark}

We make the following assumption\footnote{See Remark~\ref{rem:weakAss} below, to see how we may weaken this assumption.}.

\begin{enumerate}
\item[A1] The eigenvalues of $M_g$ and $M_e$ are non-degenerate. This is equivalent to the assumption that $\pi/\phi$ is not a rational number. This implies that the only eigenvectors of $M_g$ and $M_e$ are the Fock states $\{\ket{n}:n\in \bbZ_0^+\}$.
\end{enumerate}
The following Theorem is our main result.
\begin{theorem}\label{the:mainRes} Suppose assumption A1 is true. For any initial measure $\mu$, let $\Gamma_k(\mu) = \Gamma^{\bar{\alpha},\delta}_k(\mu)$ be the Markov flow induced by Equations~\eqref{eqn:rhohalf},~\eqref{eqn:rhoone} and control $\alpha_k$ given in Equation~\eqref{eqn:control} with Lyapunov function $V$ in Equation~\eqref{eqn:lyap}. Here, $\bar{\alpha}$ determines the control signal and $\delta$ determines the Lyapunov function.

Given any $\epsilon > 0$ and $C > 0$, there exist constants $\delta >0$ and $\bar{\alpha}$ such that for all $\mu$ satisfying $\bbE_\mu[V] \leq C$, $\Gamma_k(\mu)$ converges (weak-$\ast)$ to a limit set $\Omega$. Moreover for all $\mu_\infty\in \Omega$, $\ket{\psi} \in \supp(\mu_\infty)$ only if $\ket{\psi}$ is one of the Fock states $\ket{n}$ and
\begin{equation*}
\mu_\infty(\{\bar{n}\}) \geq 1 - \epsilon.
\end{equation*}
\end{theorem}

\subsection{Overview of Proof} 
The proof of the Theorem uses several Lemmas that are proven in the next two subsections. We outline the central idea of the proof now. 

The Lyapunov function is such that choosing $\alpha_k = 0$ ensures that the expectation value of the Lyapunov function is non-increasing. Therefore by choosing $\alpha_k$ to be the $\alpha\in [-\bar{\alpha},\bar{\alpha}]$, that minimizes the expectation value of the Lyapunov function in each step, we ensure that $V(\ket{\psi_k})$ is a super-martingale (Lemma~\ref{lem:supMart}). 

The set $\{\ket{\psi}\in \bar{B}_1:\sum_{n} \sigma_n|\iprod{\psi|n}|^2\leq C\}$ is a compact subset of $\bar{B}_1$ (Lemma~\ref{lem:comp}). We use this fact and the super-martingale property of $V(\ket{\psi_k})$ to show that with probability approaching 1, $\ket{\psi_k}$ is in a compact set for all $k$. This implies the tightness of the sequence $\Gamma_n(\mu)$. Hence, we can use Prohorov's Theorem~\ref{the:prohorov} to prove the existence of a converging subsequence (Lemma~\ref{lem:convSub}).

Suppose $\Gamma_{k_l}(\mu)$ is some subsequence that converges to a measure $\mu_\infty$. We show that the second term in the Lyapunov function $V_2(\ket{\psi}) \triangleq \delta(\cos^4(\theta_{\bar{n}})+ \sin^4(\theta_{\bar{n}})-\|M_g\ket{\psi}\|^4 - \|M_e\ket{\psi}\|^4)$ is a strict Lyapunov function for Fock states. That is $\bbE[V_2(\ket{\psi_{k+1}})|\ket{\psi_k}] - V_2(\ket{\psi_k}) \leq 0$ with equality if and only if $\ket{\psi_k}$ is a Fock state. This implies that the support set of all $\mu_\infty$ only consists of Fock states (Lemma~\ref{lem:supp}). 

Finally we note that the first part of the Lyapunov function $V_1(\ket{\psi}) = \sum_n\sigma_n|\iprod{\psi|n}|^2$ satisfies $\nabla_\alpha V_1(\ket{m}) = 0$ for all $m$. Moreover $\nabla^2_\alpha V_1(D_\alpha\ket{m}) < 0$ for $m \neq \bar{n}$ and $\nabla^2_\alpha V_1(D_\alpha\ket{\bar{n}}) > 0$. This implies $\ket{m}, m \neq \bar{n}$ is a local maximum for the Lyapunov function. Therefore, we can find a feedback $\alpha$ to move the system away from Fock states $\ket{m}$ if $m \neq \bar{n}$ with high probability (Lemma~\ref{lem:probFock}). This therefore implies that we converge to our target Fock state with high probability.

We prove in the following two sections the convergence result rigorously. We first establish some properties of the Lyapunov function.
\subsection{Properties of the Lyapunov function}
\begin{lemma}
For $\delta$ small enough $V(\ket{\psi})$ is non-negative on $\bar{B}_1$.
\end{lemma}
\begin{proof}
We have, 
\begin{eqnarray*}
\cos^4(\phi_{\bar{n}}) - \|M_g\ket{\psi}\|^4 &=& (\cos^2(\phi_{\bar{n}}) - \|M_g\ket{\psi}\|^2) (\cos^2(\phi_{\bar{n}}) + \|M_g\ket{\psi}\|^2) \\
&\geq&  \left(\cos^2(\phi_{\bar{n}}) - \sum_{n=0}^\infty \cos^2(\phi_n)|\iprod{\psi|n}|^2\right) \cos^2(\phi_{\bar{n}})\\
&\geq& -\sum_{\substack{ n= 0\\ n\neq \bar{n}}}|\iprod{\psi|n}|^2.
\end{eqnarray*}
Using a similar analysis for $\sin^4(\phi_{\bar{n}}) -  \|M_e\ket{\psi}\|^4$ we get 
\begin{equation*}
V(\ket{\psi}) \geq \sum_n (\sigma_n - 2\delta(1 - \delta_{n\bar{n}}))|\iprod{\psi|n}|^2.
\end{equation*}

Here, $\delta_{n\bar{n}}$ is the Kronecker-delta. Because 
\begin{equation*}\sigma_n\geq \min\{\sigma_{\bar{n}+1},\sigma_{\bar{n}-1}\} > 0\end{equation*} for all $n\neq \bar{n}$, for $\delta$ small enough, $V \geq 0$. 
\end{proof}
For the remainder of this paper, we assume that $\delta$ has been chosen small enough to ensure that $V$ is non-negative. Also note that for $\delta$ small enough $V(\ket{\psi})=0$ if and only if $\ket{\psi} = \ket{\bar{n}}$. Therefore $\ket{\bar{n}}$ is a strict global minimum for $V$.

\begin{lemma}\label{lem:VlowSem} The function $V$ is lower semi-continuous on $\cD(V)$. 
\end{lemma}
\begin{proof}
Because $M_g$ and $M_e$ are bounded operators, $\|M_g\ket{\psi}\|$ and $\|M_e\ket{\psi}\|$ are continuous functions on $\cD(V)$. So we only need to prove the lower semi-continuity of $\|\ket{\psi}\|_\sigma \triangleq \sum_{n=0}^\infty \sigma_n|\iprod{\psi|n}|^2$.

Let $\ket{\psi_0} = \sum_{n=0}^\infty c_n\ket{n} \in \cD(V)$. Let $\epsilon > 0$ be given. Then the finiteness of $V(\ket{\psi_0})$ implies that there exists an $N$ such that 
\begin{equation*}
\sum_{n=N}^\infty \sigma_n|\iprod{\psi|n}|^2 < \frac{\epsilon}{2}.
\end{equation*}
We can choose $\kappa$ small enough such that for all $c'_n$ satisfying $|c'_n-c_n|^2 < \kappa$, 
\begin{equation*}
\sum_{n=0}^{N-1} \sigma_n\big||c'_n|^2-|c_n|^2 \big|< \frac{\epsilon}{4}.
\end{equation*}
Therefore, for all $\ket{\psi} \in B_\kappa(\ket{\psi_0})$ 
\begin{eqnarray*}
\sum_{n=0}^{\infty} \sigma_n|\iprod{\psi|n}|^2 &\geq& \sum_{n=0}^{N-1} \sigma_n|\iprod{\psi|n}|^2\\
&>& \sum_{n=0}^{N-1} \sigma_n|\iprod{\psi_0|n}|^2 - \frac{\epsilon}{2}\\
&>& \sum_{n=0}^{\infty} \sigma_n|\iprod{\psi_0|n}|^2 - \epsilon.
\end{eqnarray*}
\end{proof}

\begin{lemma}\label{lem:comp}
For all $C \geq 0$, the set $\cV_C = \{\ket{\psi}\in \bar{B}_1: V(\ket{\psi})\leq C]\}$ is compact in $\bar{B}_1$ with respect to the topology inherited from $\cH$.
\end{lemma}
\begin{proof}
Let $\ket{\psi_1},\ket{\psi_2},\ldots$ be a sequence in $\cV_C$ and let 
\begin{equation*}
\ket{\psi_l} = \sum_{m=0}^\infty c_{l,m}\ket{m}.
\end{equation*}
Because, $\|M_g \ket{\psi}\|,\|M_e \ket{\psi}\| \leq 1$ and $\sigma_m$ is strictly increasing for $m > \bar{n}$, we have 
\begin{equation*}
\sum_{m'=m}^\infty |c_{l,m'}|^2 \leq \frac{C- 4}{\sigma_m}
\end{equation*} 
for all $l = 1,2,\ldots$ and $m > \bar{n}$. Therefore, using a diagonalization argument we know that there exists a subsequence $\ket{\psi_{l_p}}$ and a set of numbers $c_{\infty,m}$, $m \in \bbZ^+_0$ such that 
\begin{eqnarray}
&\sum_{m'=m}^\infty |c_{\infty,m'}|^2 \leq  \frac{C- 4}{\sigma_m}&\textrm{ if } m > \bar{n} \label{eqn:lcomp1a}\\
&|c_{\infty,m}|^2 \leq 1&\textrm{ if } m\leq \bar{n} \label{eqn:lcomp1b}\\
&c_{l,m}\to c_{\infty,m}.&\label{eqn:lcomp1c}
\end{eqnarray}
We claim that $\ket{\psi_{l_n}} \to \sum_{m=0}^\infty c_{\infty,m}\ket{m}$. To see this suppose $\epsilon > 0$ is given. Then because $\sigma_m\to \infty$ as $m\to\infty$, there exists an $M$ such that 
\begin{equation}\label{eqn:lcomp2}
(C-4)/\sigma_M <\epsilon/2.
\end{equation}
Also because $c_{l_p,m}\to c_{\infty,m}$ as $p \to \infty$, there exists a $P$ large enough such that for all $p\geq P$, 
\begin{equation}\label{eqn:lcomp3}
|c_{l_p,m} - c_{\infty,m}|^2 < \frac{\epsilon}{2M}\textrm{ for } m = 1,2,\ldots,M-1
\end{equation}
Combining Equations~\eqref{eqn:lcomp1a}-\eqref{eqn:lcomp3} we get for $p\geq P$
\begin{equation*}
\left\|\ket{\psi_{l_p}} - \sum_{m=0}^\infty c_{\infty,m}\ket{m}\right\| < \epsilon.
\end{equation*}
The lower semi-continuity of $V$ implies that the set $\cV_C$ is closed. Therefore $\sum_{m=0}^\infty c_{\infty,m}\ket{m} \in \cV_C$. Hence $\cV_C$ is compact.
\end{proof}

Let $\{s_n\}$ be any sequence of positive numbers with $s_n\to \infty$ as $n\to \infty$. Define an inner product
\begin{equation*}
\iprod{x|y}_s= \sum_{n=0}^\infty s_n\iprod{x|n}_\cH\iprod{n,y}_\cH.
\end{equation*}
Here the $\iprod{\cdot|\cdot}_s$ is defined on the linear subspace $\cH_s$ of $\cH$ on which $\|\cdot\|_s \triangleq \sqrt{\iprod{\cdot|\cdot}_s}$ is finite.

Given a linear operator $A$, let $R_\lambda(A) = (\lambda I - A)^{-1}$ be its resolvent operator.
We recall below a Theorem in~\cite[Th. 12.31]{Renardy2004}.
\begin{theorem}\label{the:anal} A closed, densely defined operator $A$ in some Hilbert space $X$ is the generator of an analytic semigroup $\exp(\alpha A)$ (w.r.t. the uniform topology of the set of bounded operators on $X$) if and only if there exists $\omega\in \bbR$ such that the half-plane $\Re{\lambda} > \omega$ is contained in the resolvent set of $A$ and, moreover,
there is a constant $C$ such that
\begin{equation*}
\|R_\lambda(A)\| \leq C/|\lambda - \omega|.
\end{equation*}
\end{theorem}

We show that if we consider $(a^\dagger-a)$ to be an operator on some domain in $\cH_s$ then $\exp(\alpha(a^\dagger-a))$ is an analytic semigroup on $\cH_s$.

\begin{lemma} The operator $i(a-a^\dagger)$ is symmetric on $\cH_s$.
\end{lemma}
\begin{proof}
Let $\ket{\psi},\ket{\phi}\in\cD(a-a^\dagger)\subset \cH_s$ and let $\ket{\psi} = \sum_{n=0}^\infty c_n \ket{n}$ and $\ket{\phi} = \sum_{n=0}^\infty d_n\ket{n}$.
Then, 
\begin{eqnarray*}
(a-a^\dagger)\ket{\psi} &=& \sum_{n=0}^\infty c_n(a-a^\dagger)\ket{n}\\
&=& \sum_{n=0}^\infty c_n(\sqrt{n}\ket{n-1}-\sqrt{n+1}\ket{n+1}) \\
&=& \sum_{n=0}^\infty c_{n+1} \sqrt{n+1}\ket{n} - \sum_{n=1}^\infty c_{n-1}\sqrt{n}\ket{n}
\end{eqnarray*}
Therefore,
\begin{eqnarray*}
\lefteqn{\iprod{i(a-a^\dagger)\psi |\phi}_s}\\
&=& \sum_{p=0}^\infty s_p \iprod{i(a-a^\dagger)\psi|p}_\cH\iprod{p|\phi}_\cH\\
&=& -i\sum_{p=0}^\infty s_p \left(\sum_{n=0}^\infty c_{n+1}\sqrt{n+1}\delta_{np} - \sum_{n=1}^\infty c_{n-1}\sqrt{n}\delta_{np}\right)d_p\\
&=& -i\sum_{p=0}^\infty s_p(c_{p+1}d_p- c_pd_{p+1})\sqrt{p+1}
\end{eqnarray*}
Similarly, 
\begin{equation*}
\iprod{\psi |i(a-a^\dagger)\phi}_s = i\sum_{p=0}^\infty s_p(d_{p+1}c_p- d_pc_{p+1})\sqrt{p+1}
\end{equation*} 
Therefore the operator $i(a-a^\dagger)$ is symmetric. 
\end{proof}
The operators $a^\dagger-a$ is well defined on the finite linear span of the Fock basis $\ket{n}$ and this finite linear span of the Fock basis is obviously dense in the space $\cH_s$. The fact that $a^\dagger-a$ is closable follows from the fact that $i(a-a^\dagger)$ is symmetric (see e.g.~\cite[p. 270]{Kato1966}). Therefore, we can assume that $a^\dagger-a$ is a closed operator (if it is not closed then we can always redefine $a^\dagger-a$ to be its closed linear extension).

\begin{lemma}\label{lem:anal} The semigroup $D_\alpha = \exp(\alpha(a^\dagger - a))$ is analytic  in the set of bounded operators on $\cH_s$ with respect to the uniform operator topology of bounded operators on $\cH_s$.
\end{lemma}
\begin{proof}
Because $i(a-a^\dagger)$ is symmetric, we have for all complex numbers $\lambda$  (see e.g.~\cite[p. 270]{Kato1966})
\begin{equation*}
\|i(a-a^\dagger)-\lambda\bbI\|\geq |\Im{\lambda}|.
\end{equation*}
Therefore, 
\begin{equation*}
\|R_\lambda(a^\dagger-a)\|\leq \frac{1}{|\Re{\lambda}|}.
\end{equation*}
Therefore with $\omega = 0$ and $C = 1$ and because $(a^\dagger-a)$ is closed and densely defined, the conditions of Theorem~\ref{the:anal} are satisfied. Hence the semigroup generated by $(a^\dagger-a)$ is analytic. 
\end{proof}

\begin{lemma}\label{lem:derBound}
Let $m\neq\bar{n}$ and $C > V(\ket{m})$ be given. Then there exists a constant $\bar{\epsilon} = \bar{\epsilon}(\ket{m}) > 0$ such that for all $\epsilon \in (0,\bar{\epsilon})$ and  $\ket{\psi} \in B'_\epsilon(\ket{m})$ satisfying $V(\ket{\psi}) \leq C$,  
\begin{equation}\label{eqn:diff}
V(D_\alpha \ket{\psi}) - V(\ket{\psi}) = f_1(\ket{\psi})\alpha + f_2(\ket{\psi}) \alpha^2 + O(\bar{\alpha}^3) + O(\delta)
\end{equation}
for $|\alpha| < \bar{\alpha}= \bar{\alpha}(C)$. Moreover, $f_2(\ket{\psi})< \gamma_m < 0$ for some constant $\gamma_m$.
\end{lemma}
\begin{proof}
Because $M_e$ and $M_g$ are bounded operators on $\cH$ and $D_\alpha$ is an analytic semigroup on $\cH$, we have
\begin{equation*}
V(D_\alpha \ket{\psi}) - V(\ket{\psi}) = \|D_\alpha\ket{\psi}\|^2_\sigma- \|\ket{\psi}\|^2_\sigma + O(\delta).
\end{equation*}
where $\|\cdot\|_\sigma \triangleq \sum_{n=0}^\infty \sigma_n|\iprod{\cdot|n}|^2$.

We know from Lemma~\ref{lem:anal} that $D_\alpha$ is analytic with respect to the uniform-operator topology induced by the semi-norm\footnote{Even though the Lemma~\ref{lem:anal} was proven for sequences $s_n> 0$,  we can apply it to the case where one of the $s_n$ is zero. In fact small changes in $s_n$ do not change the analycity of $D_\alpha$. We can prove the same Lemma by considering the quotient spaces of the semi-norm $\|\cdot\|_\sigma$.} $\|\cdot\|_\sigma$. Hence, for all $\ket{\psi}$ such that $\|\ket{\psi}\|_\sigma \leq V(\psi) \leq C$ we can write, using the Taylor series expansion at $\alpha = 0$, 
\begin{equation*}
\|D_\alpha\ket{\psi}\|^2_\sigma - \|\ket{\psi}\|^2_\sigma = f_1(\ket{\psi})\alpha + f_2(\ket{\psi}) \alpha^2 + O(\bar{\alpha}^3)
\end{equation*}
for $|\alpha| < \bar{\alpha}$ and $\bar{\alpha}$ only depends on $C$. 

We now prove the required bound on $f_2$ using the Baker-Campbell-Hausdorff formula~\eqref{eqn:bch}. By noting that 
\begin{equation*}
\|D_\alpha \ket{\psi}\|^2_\sigma = \sum_{n=0}^\infty\sigma_n\tr{ \exp(\alpha(a^\dagger-a))\kebr{\psi}\exp(-\alpha(a^\dagger-a))\kebr{n}}
\end{equation*}
and setting $H= a^\dagger-a$ and $A = \kebr{\psi}$ in Equation~\eqref{eqn:bch}, we get
\begin{eqnarray}
\lefteqn{f_2(\ket{\psi}) }\nonumber\\ &=&\frac{1}{2}\sum_{n=0}^\infty\sigma_n\tr{[a^\dagger-a,[a^\dagger-a,\kebr{\psi}]\kebr{n}}\nonumber\\
&=&  \frac{1}{2}\sum_{n=0}^\infty \sigma_n\tr{\left(T^2\kebr{\psi}+ \kebr{\psi}T^2 - 2T\kebr{\psi} T\right)\kebr{n}}\label{eqn:der1}.
\end{eqnarray}
Where we have set $T = a^\dagger-a$. If we let  $\ket{\psi} = \sum_n c_n \ket{n}$, then 
\begin{eqnarray}
\lefteqn{\frac{1}{2}\tr{\left(T^2\kebr{\psi}+\kebr{\psi}T^2\right)\kebr{n}}}\nonumber\\
&=&\frac{1}{2}{\bra{n}T^2\ket{\psi}\iprod{\psi|n} + \bra{\psi}T^2\ket{n}\iprod{n,\psi}}\nonumber\\ 
&=& \Re{\sqrt{(n+1)(n+2)}c_{n+2}c_n^\ast + \sqrt{n(n-1)}c_{n-2}c_n^\ast} \nonumber\\&&- (2n+1)|c_n|^2\label{eqn:der2}
\end{eqnarray}
and
\begin{eqnarray}
\lefteqn{\tr{T\kebr{\psi}T\kebr{n}}}\nonumber\\&=& \bra{n}T\ket{\psi}\cdot \bra{\psi} T\ket{n}\nonumber\\
&=& \sqrt{n(n+1)}\Re{c_{n-1}c_{n+1}^\ast} -n|c_{n-1}|^2 - (n+1)|c_{n+1}|^2.\label{eqn:der3}
\end{eqnarray}
Substituting Equations~\eqref{eqn:der2} and~\eqref{eqn:der3} into~\eqref{eqn:der1} and rearranging terms, we get
\begin{eqnarray}
f_2(\ket{\psi}) &=& \sum_{n=0}^\infty |c_n|^2\big((n+1)\sigma_{n+1} + n\sigma_{n-1} - (2n+1)\sigma_n\big)\nonumber\\
&&+\: \sum_{n=0}^\infty \Re{c_{n-1}c_{n+1}^\ast}\sqrt{n(n+1)}(\sigma_{n-1} + \sigma_{n+1} - 2\sigma_n)\label{eqn:f2}
\end{eqnarray}

If $2\leq n\neq \bar{n}$ then substituting for $\sigma_n$ from Equation~\eqref{eqn:sigmaDef} we get
\begin{equation*}
(n+1)\sigma_{n+1} + n\sigma_{n-1} - (2n+1)\sigma_n = \frac{-1}{n(n+1)}
\end{equation*}
and for $n= 0,1$ we get 
\begin{equation*}
(n+1)\sigma_{n+1} + n\sigma_{n-1} - (2n+1)\sigma_n = \frac{-1}{4}
\end{equation*}
Substituting this into~\eqref{eqn:f2}, we have for $2\leq m\neq \bar{n}$,
\begin{eqnarray}
f_2(\ket{\psi}) &\leq& -\frac{|c_m|^2}{\kappa_m} +|c_{\bar{n}}|^2 ((\bar{n}+1)\sigma_{\bar{n}+1} + \bar{n}\sigma_{\bar{n}-1})\nonumber\\
&&+\: \sum_{n=0}^\infty \Re{c_{n-1}c_{n+1}^\ast}\sqrt{n(n+1)}(\sigma_{n-1} + \sigma_{n+1} - 2\sigma_n)\label{eqn:der4}
\end{eqnarray}
where 
\begin{equation*}
\kappa_m = \left\{\begin{array}{ll} {m(m+1)}&\textrm{if } 2\leq m\neq \bar{n,}\\ 4&\textrm{if } m = 1,2\end{array}\right.
\end{equation*}

For all $\ket{\psi} \in B'_\epsilon(\ket{m})$, the first term in Equation~\eqref{eqn:der4} will be less than $-3/(4\kappa_m)$ for $\epsilon$ small enough. We show that for $\epsilon$ small enough, for all $\ket{\psi}\in B'_\epsilon(\ket{m})$ such that $V(\ket{\psi}) \leq C$,
\begin{eqnarray}
\frac{1}{2\kappa_m} &\geq& |c_{\bar{n}}|^2 ((\bar{n}+1)\sigma_{\bar{n}+1} + \bar{n}\sigma_{\bar{n}-1})\nonumber\\ &&+ \left|\sum_{n=0}^\infty \Re{c_{n-1}c_{n+1}^\ast}\sqrt{n(n+1)}(\sigma_{n-1} + \sigma_{n+1} - 2\sigma_n)\right|\label{eqn:der5}
\end{eqnarray}
For all $\ket{\psi}\in B'_\epsilon(\ket{m})$ we have $|c_{\bar{n}}|^2 < \epsilon^2$ and hence we can choose a small enough $\epsilon$ such that
\begin{equation}\label{eqn:der5a}
 |c_{\bar{n}}|^2 ((\bar{n}+1)\sigma_{\bar{n}+1} + \bar{n}\sigma_{\bar{n}-1}) \leq\frac{1}{8\kappa_m}.
\end{equation}

We have
\begin{eqnarray*}
|\sqrt{n(n+1)}(\sigma_{n-1} + \sigma_{n+1} - 2\sigma_n)| &=& \sqrt{n(n+1)}\frac{n^2+3n+1}{n^2(n+1)^2}\\
&\triangleq& h(n)
\end{eqnarray*}
The function $h(n)$ defined in the above Equation  is of order $O(1/n)$. Because $\sigma(n)$ is order $O(\ln(n))$ and because the series $\sum_n|c_n|^2\sigma_n < C$ converges, we know that the series $\sum_n|c_n|^2h(n)$ converges. Hence, $\sum_{n=M}^\infty |c_n|^2h(n)$ can be made arbitrarily small for large $M$. By Cauchy-Schwartz, $\sum_n|c_n|^2 \geq \sum_n\Re{c_{n-1}c_{n+1}}$ and hence $M$ can be chosen large enough such that for all $\ket{\psi}$ satisfying $V(\ket{\psi}) \leq C$, we have
\begin{equation}\label{eqn:der5b}
\left|\sum_{n=M}^\infty \Re{c_{n-1}c_{n+1}^\ast}\sqrt{n(n+1)}(\sigma_{n-1} + \sigma_{n+1} - 2\sigma_n)\right|\leq \frac{1}{8\kappa_m}.
\end{equation}
Because  $\sum_n|c_n|^2 \geq \sum_n\Re{c_{n-1}c^\ast_{n+1}}$, we can choose $\epsilon$ small enough so that for all $\ket{\psi}\in B_\epsilon(\ket{m})$, 
\begin{equation}\label{eqn:der5c}
\left|\sum_{n=0}^{M-1} \Re{c_{n-1}c_{n+1}^\ast}\sqrt{n(n+1)}(\sigma_{n-1} + \sigma_{n+1} - 2\sigma_n)\right|\leq \frac{1}{4\kappa_m}.
\end{equation}
Equations~\eqref{eqn:der5a}, \eqref{eqn:der5b} and~\eqref{eqn:der5c} prove Equation~\eqref{eqn:der5}.  Substituting~\eqref{eqn:der5} into~\eqref{eqn:der4} we get the required bound,
\begin{equation*}
f_2(\ket{\psi}) \leq \gamma_m < 0
\end{equation*}
where $\gamma_m = -1/(4\kappa_m)$.
\end{proof}
We have the following Corollary to the above Lemma.
\begin{corollary}\label{cor:minStep} Given any $m\neq \bar{n}$ and $C > V(\ket{m})$, the constants $\bar{\alpha}$ and $\delta$ in Equations~\eqref{eqn:control} and~\eqref{eqn:lyap}, respectively, can be chosen to be small enough such that there exists an $\epsilon> 0$ and $c > 0$ such that 
if 
\begin{equation*}
\ket{\psi_k} \in \cW_\epsilon \triangleq B'_\epsilon(\ket{m})\cap \{\ket{\psi}:C \geq V(\ket{\psi}) > V(\ket{m}) - \epsilon\}
\end{equation*}
then
\begin{equation*}
V(\ket{\psi_{K+1}}) - V(\ket{\psi_k}) < -c,\textrm{ with probability 1.}
\end{equation*}
\end{corollary}
\begin{proof} 
Let $\bar{\epsilon}$ be as in Lemma~\ref{lem:derBound} and choose $\epsilon'\in (0,\bar{\epsilon})$. Because $M_g$ and $M_e$ are bounded operators on $\cH$ and $M_g\ket{m}= \ket{m}$ and $M_e\ket{m} = \ket{m}$, $\epsilon$ can be chosen small enough such that if $\ket{\psi_k}\in B'_\epsilon(\ket{m})$ then $\ket{\psi_{k+1/2}} \in B'_{\epsilon'}(\ket{m})$. Moreover, we can show using the same argument used to prove Lemma~\ref{lem:supMart} below, we know that if $V(\ket{\psi_k}) < C$ then $V(\ket{\psi_{k+1/2}}) < C$. Therefore, if $\ket{\psi_{k}}\in \cW_{\epsilon}$ then $\ket{\psi_{k+1/2}} \in \cW_{\epsilon'}$ with probability $1$. 

Now we choose $\bar{\alpha}$ small enough such that the term of order $O(\bar{\alpha}^3)$ in Equation~\eqref{eqn:diff} is smaller than $\frac{\gamma_m\bar{\alpha}^2}{2}$. Then from Lemma~\ref{lem:derBound} we know that for all $\ket{\psi}\in \cW_{\epsilon'}$ and for either $\alpha = +\bar{\alpha}$ or $\alpha = -\bar{\alpha}$, we have
\begin{equation*}
V(D_\alpha\ket{\psi}) - V(\ket{\psi}) \leq -\frac{\gamma_m\bar{\alpha}^2}{2}.
\end{equation*}
We can choose $\delta$ small enough such that the $O(\delta)$ term in Lemma~\ref{lem:derBound} is smaller than $\frac{\gamma_m\bar{\alpha}^2}{4}$. Therefore, for all $\ket{\psi_k}\in \cW_\epsilon$ the statement of the Corollary is satisfied by choosing $c = \frac{\gamma_m\bar{\alpha}^2}{4}$.
\end{proof}
\subsection{Proof of Convergence}
We wish to establish that $V(\ket{\psi_k})$ is super-martingale. We have,
\begin{eqnarray}
\lefteqn{\bbE\left[V(\ket{\psi_{k+1}})\big|\ket{\psi_k}= \ket{\psi}\right] - V(\ket{\psi})} \nonumber \\
&=& \bbE\left[V\left(D_{\alpha_k}\left(\ket{\psi_{k+1/2}}\right)\right)\big|\ket{\psi_k} = \ket{\psi}\right]  - V(\ket{\psi})\nonumber\\
&=& \bbE\left[\argmin_{\alpha\in [-\bar{\alpha},\bar{\alpha}]} V\left(D_{\alpha}\left(\ket{\psi_{k+1/2}}\right)\right)\bigg|\ket{\psi_k} = \ket{\psi}\right]　 - V(\ket{\psi})\nonumber\\
&=& K_1(\ket{\psi}) + K_2(\ket{\psi})\label{eqn:condExeV}
\end{eqnarray}
where we set 
\begin{eqnarray}
K_1(\ket{\psi}) &\triangleq& \bbE\left[\argmin_{\alpha\in [-\bar{\alpha},\bar{\alpha}]} V\left(D_{\alpha}\left(\ket{\psi_{k+1/2}}\right)\right)\big|\ket{\psi_k} = \ket{\psi}\right] \nonumber\\
&& -\: \bbE\left[V(\ket{D_0(\psi_{k+1/2}})\big|\ket{\psi_k}= \ket{\psi}\right],\label{eqn:k1}\\
K_2(\ket{\psi}) &\triangleq&  \bbE\left[V\left(D_{0}\left(\ket{\psi_{k+1/2}}\right)\right)\big|\ket{\psi_k} = \ket{\psi}\right] - V(\ket{\psi})\nonumber\\
&=&  \bbE\left[V\left(\ket{\psi_{k+1/2}}\right)\big|\ket{\psi_k} = \ket{\psi}\right] - V(\ket{\psi}) \label{eqn:k2}
\end{eqnarray}
Obviously $K_1(\ket{\psi})$ is non-positive for all $\ket{\psi}$. In order to calculate $K_2(\ket{\psi})$ we set $\rho = \rho_\psi = \kebr{\psi}$ and we note that $M_g$ and $M_e$ commute and satisfy $M_g^2 + M_e^2 = \bbI$, where $\bbI$ is the identity operator.
Therefore,
\begin{eqnarray*}
\|M_e\ket{\psi}\|^4 &=& \tr{M_e^2\rho}^2 \\
&=& \left(\tr{(\bbI- M_g^2)\rho}\right)^2 \\
&=& \left(1 - \tr{M_g^2\rho}\right)^2 \\
&=& 1 - 2\tr{M_g^2\rho} + \tr{M_g^2\rho}^2\\
&=& 1 - 2\|M_g\ket{\psi}\|^2 + \|M_g\ket{\psi}\|^4
\end{eqnarray*}

But, 
\begin{eqnarray*}
\lefteqn{\bbE \left[\|M_g\ket{\psi_{k+1/2}}\|^2 \big| \ket{\psi_k} = \psi\right]}\\
&=& \|M_g\ket{\psi}\|^2 \left\|M_g\frac{M_g\ket{\psi}}{\|M_g\ket{\psi}\|}\right\|^2 + \|M_e\ket{\psi}\|^2 \left\|M_g\frac{M_e\ket{\psi}}{\|M_e\ket{\psi}\|}\right\|^2\\
&=& \tr{M_g^4\rho} + \tr{M_g^2 M_e^2 \rho}\\
&=& \tr{M_g^2\rho}\\
&=& \|M_g\ket{\psi_k}\|^2.
\end{eqnarray*}
Similarly, we get for all $\ket{n}$,
\begin{equation*}
\bbE[|\iprod{n,{\psi_{k+1/2}}}|^2\big|\psi_k = \ket{\psi}] = |\iprod{n,{\psi_{k}}}|^2
\end{equation*}

Therefore, 
\begin{equation*}
K_2(\ket{\psi}) = 2\left(\|M_g\ket{\psi}\|^4 - \bbE\left[\|M_g\ket{\psi_{k+1/2}}\|^4 \big| \ket{\psi_{k}} = \ket{\psi}\right]\right)
\end{equation*}
We have,
\begin{eqnarray*}
\lefteqn{\bbE\left[\|M_g\ket{\psi_{k+1/2}}\|^4 \big| \ket{\psi_{k}} = \ket{\psi}\right]}\\
&=& \tr{M_g^2\rho}\tr{M_g^2 \frac{M_g\rho M_g}{\tr{M_g^2\rho}}}^2 + \tr{M_e^2\rho}\tr{M_g^2 \frac{M_e\rho M_e}{\tr{M_e^2\rho}}}^2\\
&=& \frac{\tr{M_g^4\rho}^2}{\tr{M_g^2\rho}} + \frac{\tr{M_g^2M_e^2\rho}^2}{\tr{M_e^2\rho}}\\
&=& \frac{\tr{M_g^4\rho}^2 - \tr{M_g^2\rho}\tr{M_g^4\rho}^2 + \tr{M_g^2\rho}\left(\tr{M_g^2\rho} - \tr{M_g^4\rho}\right)^2}{\tr{M_g^2\rho}\tr{M_e^2\rho}} \\
&=&\frac{\tr{M_g^4\rho}^2 + \tr{M_g^2\rho}^3 - 2\tr{M_g^2\rho}^2\tr{M_g^4\rho}}{\tr{M_g^2\rho}\tr{M_e^2\rho}} 
\end{eqnarray*}

Therefore,
\begin{eqnarray}
K_2(\ket{\psi}) &=& 2\left(\|M_g\ket{\psi}\|^4 - \bbE\left[\|M_g\ket{\psi_{k+1}}\|^4 \big| \ket{\psi_{k}} = \ket{\psi}\right]\right)\nonumber\\
&=& 2\left(\tr{M_g^2\rho}^2 - \frac{\tr{M_g^4\rho}^2 + \tr{M_g^2\rho}^3 - 2\tr{M_g^2\rho}^2\tr{M_g^4\rho}}{\tr{M_g^2\rho}\tr{M_e^2\rho}} \right)\nonumber\\
&=& \frac{-2\left(\tr{M_g^4\rho} - \tr{M_g^2\rho}^2\right)^2}{\tr{M_g^2\rho}\tr{M_e^2\rho}}\label{eqn:K2Neg}
\end{eqnarray}

The following Lemma is a consequence of the above calculation.
\begin{lemma}\label{lem:supMart} $V(\ket{\psi_k})$ is a super-martingale and satisfies
\begin{equation*}
\bbE\left[V(\ket{\psi_{k+1}})\big|\ket{\psi_k}= \ket{\psi}\right] - V(\ket{\psi})  \leq -2\frac{\left(\tr{M_g^4\rho} - \left(\tr{M_g^2\rho}\right)^2\right)^2}{\tr{M_g^2\rho}\tr{M_e^2\rho}}.
\end{equation*}
Here $\rho = \ket{\psi}\bra{\psi}$.
\begin{proof}
Firstly from the lower semi-continuity of $V$ (Lemma~\ref{lem:VlowSem}) we know that $V$ is measurable function of $\ket{\psi}$. Therefore $\bbE_\mu[V]$ is well defined. The super-martingale property follows from Equations~\eqref{eqn:condExeV}, \eqref{eqn:K2Neg} and the fact that $K_1(\ket{\psi})$ is not positive because $0\in [-\bar{\alpha},\bar{\alpha}]$.
\end{proof}
\end{lemma}

\begin{lemma}\label{lem:convSub}
If $\bbE_\mu[V] < \infty$ then the sequence $\Gamma_n(\mu)$ has a (weak-$\ast$) converging subsequence. 
\end{lemma}
\begin{proof}
We prove that the sequence of measures $\Gamma_n(\mu)$ is tight. 

Let $\epsilon > 0$ be given. We need to show that there exits a compact set $K_\epsilon$ such that $[\Gamma_n(\mu)](K_\epsilon) \geq 1-\epsilon$ for all $n$. By Lemma~\ref{lem:comp}, the set 
\begin{equation*}
\cV = \left\{\ket{\psi}\in \bar{B}:V(\ket{\psi}) \leq  \frac{\bbE_\mu[V]}{\epsilon}\right\}
\end{equation*}
is compact. We prove that $[\Gamma_n(\mu)](\cV) \geq 1-\epsilon$. 

Because $V(\ket{\psi_k})$ is a super-martingale, we have 
\begin{equation*}
\bbE_{\Gamma_n(\mu)}[V] \leq \bbE_{\mu}[V].
\end{equation*}
Hence, by applying Doob's inequality, the probability that $V(\ket{\psi_n}) > \bbE_\mu[V]/{\epsilon}$ is
\begin{equation*}
[\Gamma_n(\mu)](\bar{B}_1\setminus \cV) < \epsilon.
\end{equation*}

Hence, the sequence $\Gamma_n(\mu)$ is tight and by Prohorov's Theorem~\ref{the:prohorov} has a (weak-$\ast$) converging subsequence. 
\end{proof}

Let $\Omega$ denote the limit set of $\Gamma_n(\mu)$. i.e.
\begin{equation}\label{eqn:omega}
\Omega = \{\mu_\infty\in \cM_1: \Gamma_{n_m}(\mu)\wto \mu,\textrm{ for some subsequence $\Gamma_{n_m}$ of $\Gamma_n$}\}.
\end{equation}

\begin{lemma}\label{lem:supp} Suppose assumption $A1$ is true. If $\ket{\psi}$ is not a Fock state $\ket{n}$ for some $n\in \bbZ_0^+$ then $\ket{\psi}\notin \supp(\mu_\infty)$ for all $\mu_\infty\in \Omega$.
\end{lemma}
\begin{proof}
Suppose $\Gamma_{k_p}[\mu]\wto \mu_\infty$. We know from Equation~\eqref{eqn:condExeV} and Lemma~\ref{lem:supMart} that $K_1(\ket{\psi_k}) + K_2(\ket{\psi_k}) \to 0$ as $k\to \infty$. Because both $K_1$ and $K_2$ are non-positive, we need $K_2(\ket{\psi_k}) \to 0$ as $k\to \infty$. Therefore, for all initial distributions $\mu$
\begin{equation*}
\lim_{p\to\infty} \bbE_{\Gamma_{k_p}(\mu)} [K_2]  = 0.
\end{equation*}

Also note that the function $K_2$ in Equation~\eqref{eqn:K2Neg} is a continuous function of $\ket{\psi}\in \bar{B}_1$ with respect to the topology inherited from $\cH$. Therefore, by the definition of (weak-$\ast$) convergence of measures~\ref{defn:convMeas}, we have 
\begin{equation}\label{eqn:lemSupCont}
\bbE_{\mu_\infty} [K_2] = \lim_{p\to\infty} \bbE_{\Gamma_{k_p}(\mu)} [K_2]  = 0.
\end{equation}

But from Equation~\eqref{eqn:K2Neg}, we know that if $K_2(\ket{\psi}) = 0$ then 
\begin{equation*}
\tr{M_g^4\rho} = \tr{M_g^2\rho}^2.
\end{equation*}
But the Cauchy-Schwartz inequality implies $\tr{M_g^4\rho} = \tr{M_g^4\rho}\tr{\rho} \geq \tr{M_g^2\rho}^2$ with equality if and only if $M_g^4\rho$ and $\rho$ are co-linear. That is if and only if $\rho$ is a projection over the eigenstate of $M_g^4$. Therefore, by assumption $A1$, $K_2(\ket{\psi}) \leq 0$ with equality if and only if $\ket{\psi}$ is a Fock state. 

Now if $\ket{\psi}$ is not a Fock state then there is an open neighborhood $W$ of $\ket{\psi}$ such that the neighborhood does not contain a Fock state. Therefore, $\mu_\infty(W) = 0$ otherwise $\bbE_{\mu_\infty}[K] < 0$ contradicting~\eqref{eqn:lemSupCont}. Therefore $\ket{\psi}\notin\supp{\mu_\infty}$. 
\end{proof}

\begin{lemma}\label{lem:probFock} Given the assumptions of Theorem~\ref{the:mainRes}, for any Fock state $\ket{m}$, $m\neq \bar{n}$ and for all $\kappa, C> 0$, there exist constants $\bar{\alpha}$ and $\delta > 0$ and a neighborhood $\cV$ of $\ket{m}$ such that if $\bbE_\mu[V] \leq C$ then $\mu_\infty[\cV] < \kappa$ for all $\mu_\infty\in \Omega$. 
\end{lemma}
\begin{proof} 
Suppose $\Gamma_{p_q}(\mu)$ is some subsequence of $\Gamma_{p}(\mu)$ that converges to $\mu_\infty\in \Omega$. 

Given any $\nu > 0$, the lower semi-continuity of $V$ implies that the set
\begin{equation*}
\cV_\nu \triangleq B'_\nu(\ket{m}) \cap \{\ket{\psi}\in \bar{B}_1:V(\ket{\psi}) > V(\ket{m}) - \nu\}
\end{equation*}
is an open neighborhood of $\ket{m}$ in $\bar{B}_1$. Because $\cV_\nu$ is open, we know from the (weak-$\ast$) convergence of $\Gamma_{p_q}(\mu) \to \mu_\infty$ that
\begin{equation*}
\liminf_{q\to\infty} [\Gamma_{p_q}(\mu)](\cV_\nu) \geq \mu_\infty(\cV_\nu).
\end{equation*} 
Therefore, in order to prove the Lemma, we need to show that for some $\nu > 0$, 
\begin{equation*}
\lim_{q\to\infty}\Gamma_{p_q}(\cV_\nu) \leq  \kappa.
\end{equation*}
Because $\bbE_\mu[V] \leq C$ we know from Doob's inequality~\ref{eqn:doob} that for all $p$ and all $\nu > 0$,
\begin{equation}\label{eqn:fockLem1}
[\Gamma_p(\mu)]\left(\left\{\ket{\psi}\in \cV_\nu\subset\bar{B}_1: V(\ket{\psi}) > \frac{C}{\kappa/2}\right\}\right) < \frac{\kappa}{2}.
\end{equation}
We show that there exists a $\nu > 0$ such that $\lim_{p\to \infty}  [\Gamma_p(\mu)](\bar{\cV}_\nu) = 0$, where 
\begin{equation}\label{eqn:fockLem2}
\bar{\cV}_\nu \triangleq\left(\left\{\ket{\psi}\in \cV_\nu\subset\bar{B}_1: V(\ket{\psi})\leq \frac{C}{\kappa/2}\right\}\right).
\end{equation}

We know from Corollary~\ref{cor:minStep} that there exists an $\epsilon> 0$ and $c > 0$ such that 
if 
\begin{equation*}
\ket{\psi_p} \in \cW_\epsilon \triangleq B'_\epsilon(\ket{m})\cap \left\{\ket{\psi}:\frac{C}{\kappa/2} \geq V(\ket{\psi}) > V(\ket{m}) - \epsilon\right\}
\end{equation*}
then
\begin{equation*}
V(\ket{\psi_{p+1}}) - V(\ket{\psi_p}) < -c,\textrm{ with probability 1.}
\end{equation*}
Therefore if $\ket{\psi_p}\in \cW_\epsilon$, then for some finite $P$ satisfying $p< P \leq \left\lceil \left(\frac{C}{\kappa/2} - \epsilon\right)/c\right\rceil$, we have
\begin{equation*}
\ket{\psi_P}\in \{\ket{\psi}:V(\ket{\psi}) \leq \cV(\ket{m}) - \epsilon\}\textrm{ with probability $1$.}
\end{equation*}
We choose $\nu = \epsilon/2$. Then if $\ket{\psi_p} \in \bar{\cV}_\nu$ then within a finite number of steps less than $\left\lceil \left(\frac{C}{\kappa/2} - \epsilon\right)/c\right\rceil$ the system state is in the set $\{\ket{\psi}:V(\ket{\psi}) \leq \cV(\ket{m}) - \epsilon\}$. But because $\{\ket{\psi}:V(\ket{\psi}) \leq V(\ket{m}) - \epsilon\} \cap\bar{\cV}_\nu$ is empty, we know that if $\ket{\psi_p}\in \bar{\cV}_\nu$ then the process is outside $\bar{\cV}_\nu$ within a finite number of steps less than $\left\lceil \left(\frac{C}{\kappa/2} - \epsilon\right)/c\right\rceil$.

So $\lim_{p\to \infty} [\Gamma_p(\mu)](\bar{\cV}_\nu) \neq 0$ if and only if the Markov process jumps back and forth between the sets $\{\ket{\psi}:V(\ket{\psi}) \leq \cV(\ket{m}) - \epsilon\}$ and $\bar{\cV}_\nu$ infinitely often. But the supermartigale property of $V(\ket{\psi_p})$ and Doob's inequality~\ref{eqn:doob} implies  
\begin{align}
\mathrm{prob} &\left[V(\ket{\psi_p}) < V(\ket{m}) - \epsilon \textrm{ and } \inf_{p'> p} V(\ket{\psi_p'}) \geq V(\ket{m}) - \frac{\epsilon}{2} \right]\nonumber & \\ &<1 - \frac{V(\ket{m}) - \epsilon}{V(\ket{m}) - \epsilon/2} < 1.\nonumber
\end{align}
As the probability of a single jump is less than $1$, the probability of infinitely many jumps is zero. Therefore $\lim_{p\to \infty} [\Gamma_p(\mu)](\bar{\cV}_\nu) = 0$ and from Equations~\eqref{eqn:fockLem1} and~\eqref{eqn:fockLem2} we get 
\begin{equation*}
\mu_\infty(\cV_\nu) \leq \liminf_{q\to \infty}[\Gamma_{p_q}(\mu)](\cV_\nu) \leq \kappa.
\end{equation*}
\end{proof}
We now finally prove Theorem~\ref{the:mainRes}.
\begin{proof}[Proof of Theorem~\ref{the:mainRes}]
Let $\epsilon > 0$ and $C > 0$ be given. 

We know from Lemma~\ref{lem:supp} that the support set of $\mu_\infty$ only consists of the Fock states $\ket{m}$. Because $\sigma_n \to \infty$, there exists an $M$ such that $\sigma_M > C/(\epsilon/2) + \delta$. Because for all $m > M$, $\cV(\ket{m}) > \sigma_m-\delta$, the supermartingale property of $V(\ket{\psi_k})$ and Doob's inequality~\eqref{eqn:doob} implies for all $k\in \bbZ_0^+$,
\begin{equation*}
\Gamma_k(\mu)(\{\ket{m}:m\geq M\}) < \frac{\epsilon}{2}.
\end{equation*}

If we set $\kappa = \epsilon/2M$ in Lemma~\ref{lem:probFock}, then we know there exist constants $\bar{\alpha}> 0$ and $\delta> 0$ and neighborhoods $\cV(\ket{m})$ of $\ket{m}$ for $0\leq m < M, m\neq \bar{n}$, such that 
\begin{equation*}
\mu_\infty(\{\ket{m}\}) < \frac{\epsilon}{2M}.
\end{equation*}
Therefore, $\mu_\infty(\ket{\bar{n}}) = 1 - \mu_\infty(\{\ket{m}:m\neq \bar{n}\}) \geq 1- \epsilon$.
\end{proof}

\begin{remark}\label{rem:weakAss}
In Lemma~\ref{lem:supp} we show that the only vectors in the support of $\mu_\infty$ are those corresponding to eigenvector of $M_s$. We then used assumption $A1$ in the proof of the Lemma to claim that the only eigenvectors of $M_s$ are the Fock states. We can however weaken this assumption to the following: for some large $M$ such that $\sigma_M > 2C/\epsilon$, eigenvalues corresponding to eigenvectors $\ket{m}$, $m< M$ are non-degenerate. This is because, we can show that if some eigenvector $\ket{\psi}$ is in the span of the set $\{\ket{M},\ket{M+1},\ldots\}$ then using the same argument as that used for $\ket{m}, m> M$, we can show that the probability of $\ket{\psi}$ is small. This is significant for cases where $M_g$ is a more complicated non-linear function of $N$, as is the case in a practical system.
\end{remark}

\begin{remark} In this paper we prove (weak-$\ast$) convergence of $\Gamma_k(\mu)$ to some $\mu_\infty$. The quantum expectation value of an observable (self-adjoint linear operators) $T:\cH\to \cH$ is some state $\ket{\psi}$, defined $\iprod{T}_{\ket{\psi}} \triangleq \iprod{\psi|T\psi}$ is a continuous function of $\ket{\psi}$ if $T$ is bounded. Therefore, if we think of $\iprod{T}:{\ket{\psi}}\mapsto \iprod{T}_{\ket{\psi}}$ as a random variable on the state space $\bar{B}_1$, then by the  (weak-$\ast$) convergence of $\Gamma_{k_l}(\mu)$ to $\mu_\infty$, we have
\begin{equation*}
\lim_{l\to\infty}\bbE_{\Gamma_{k_l}(\mu)}[\iprod{T}] = \bbE_{\mu_\infty}[ \iprod{T}].
\end{equation*}
In particular if $T = \kebr{\bar{n}}$ then 
\begin{equation*}
\lim_{l\to\infty}\bbE_{\Gamma_{k_l}(\mu)}[\iprod{\kebr{\bar{n}}}] = \bbE_{\mu_\infty}[ \iprod{\kebr{\bar{n}}}]\geq 1-\epsilon,
\end{equation*}
where $\epsilon$ is as in Theorem~\ref{the:mainRes}. But $\iprod{\kebr{\bar{n}}}$ is precisely the (quantum) probability that the system is in state $\ket{\bar{n}}$. Therefore, the probability that the quantum system is in the state $\ket{\bar{n}}$ maybe made arbitrarily close to $1$. 

Similar statements maybe made about the standard deviation of bounded observables $T:\cH\to\cH$.
\end{remark}

Simulations have been performed with the controller in Theorem~\ref{the:mainRes} by truncating the controller using a Galerkin approximation. The simulations indicate that the controller designed using the infinite dimensional Hilbert space provides performance improvements (of about 4-5\%) in the probability of convergence to the target state when compared to the controller designed using the finite dimensional approximation~\cite{Somaraju2011}. Moreover as shown in Theorem~\ref{the:mainRes} the feedback parameters $\bar{\alpha}$ and $\delta$ maybe chosen to increase the probability of convergence to the target state.
\section{Conclusion} In this paper we examine the semi-global, approximate stabilization of the Markov process in Equations~\eqref{eqn:rhohalf} and~\eqref{eqn:rhoone} at a photon-number target state. The Markov process is defined on the set of all unit vectors in the infinite dimensional Hilbert space $\cH$. The non-compactness of this set of unit vectors dictates the use of a special Lyapunov function~\eqref{eqn:lyap} to show the following in Theorem~\ref{the:mainRes} - provided the initial measures $\mu$ satisfies certain initial conditions, for all $\epsilon > 0$ we can choose feedback such that with probability greater than $1-\epsilon$, the Markov process converges to the target Fock state. 
\section{Acknowledgments}
The authors thank M. Brune, I.~Dotsenko, S.~Haroche and J.M. Raimond  for enlightening discussions and advices. 
\bibliographystyle{model1a-num-names}
\bibliography{ref}

\begin{thebibliography}{20}
\expandafter\ifx\csname natexlab\endcsname\relax\def\natexlab#1{#1}\fi
\providecommand{\bibinfo}[2]{#2}
\ifx\xfnm\relax \def\xfnm[#1]{\unskip,\space#1}\fi
\bibitem[{Mirrahimi et~al.(2010)Mirrahimi, Dotsenko, and
  Rouchon}]{Mirrahimi2010}
\bibinfo{author}{M.~Mirrahimi}, \bibinfo{author}{I.~Dotsenko},
  \bibinfo{author}{P.~Rouchon}, in: \bibinfo{booktitle}{48th IEEE Conference on
  Decision and Control}, \bibinfo{publisher}{IEEE},
  \bibinfo{address}{Shanghai}, \bibinfo{year}{2010}, pp. \bibinfo{pages}{1451
  -- 1456}.
\bibitem[{Dotsenko et~al.(2009)Dotsenko, Mirrahimi, Brune, Haroche, Raimond,
  and Rouchon}]{Dotsenko2009}
\bibinfo{author}{I.~Dotsenko}, \bibinfo{author}{M.~Mirrahimi},
  \bibinfo{author}{M.~Brune}, \bibinfo{author}{S.~Haroche},
  \bibinfo{author}{J.-M. Raimond}, \bibinfo{author}{P.~Rouchon},
  \bibinfo{journal}{Physical Review A} \bibinfo{volume}{80}
  (\bibinfo{year}{2009}) \bibinfo{pages}{013805--013813}.
\bibitem[{Del\'{e}glise et~al.(2008)Del\'{e}glise, Dotsenko, Sayrin, Bernu,
  Brune, Raimond, and Haroche}]{Deleglise2008}
\bibinfo{author}{S.~Del\'{e}glise}, \bibinfo{author}{I.~Dotsenko},
  \bibinfo{author}{C.~Sayrin}, \bibinfo{author}{J.~Bernu},
  \bibinfo{author}{M.~Brune}, \bibinfo{author}{J.-M. Raimond},
  \bibinfo{author}{S.~Haroche}, \bibinfo{journal}{Nature} \bibinfo{volume}{455}
  (\bibinfo{year}{2008}) \bibinfo{pages}{510--514}.
\bibitem[{Gleyzes et~al.(2007)Gleyzes, Kuhr, Guerlin, Bernu, Del\'{e}glise,
  Hoff, Brune, Raimond, and Haroche}]{Gleyzes2007}
\bibinfo{author}{S.~Gleyzes}, \bibinfo{author}{S.~Kuhr},
  \bibinfo{author}{C.~Guerlin}, \bibinfo{author}{J.~Bernu},
  \bibinfo{author}{S.~Del\'{e}glise}, \bibinfo{author}{U.~B. Hoff},
  \bibinfo{author}{M.~Brune}, \bibinfo{author}{J.-M. Raimond},
  \bibinfo{author}{S.~Haroche}, \bibinfo{journal}{Nature} \bibinfo{volume}{446}
  (\bibinfo{year}{2007}) \bibinfo{pages}{297--300}.
\bibitem[{Guerlin et~al.(2007)Guerlin, J.~Bernu, Sayrin, Gleyzes, Kuhr, Brune,
  Raimond, and Haroche}]{Guerlin2007}
\bibinfo{author}{C.~Guerlin}, \bibinfo{author}{S.~D. J.~Bernu},
  \bibinfo{author}{C.~Sayrin}, \bibinfo{author}{S.~Gleyzes},
  \bibinfo{author}{S.~Kuhr}, \bibinfo{author}{M.~Brune}, \bibinfo{author}{J.-M.
  Raimond}, \bibinfo{author}{S.~Haroche}, \bibinfo{journal}{Nature}
  \bibinfo{volume}{448} (\bibinfo{year}{2007}) \bibinfo{pages}{889--893}.
\bibitem[{Haroche and Raimond(2006)}]{Haroche2006}
\bibinfo{author}{S.~Haroche}, \bibinfo{author}{J.~Raimond},
  \bibinfo{title}{Exploring the Quantum: Atoms, Cavities and Photons},
  \bibinfo{publisher}{Oxford University Press}, \bibinfo{year}{2006}.
\bibitem[{Brune et~al.(1992)Brune, Haroche, Raimond, Davidovich, and
  Zagury}]{Brune1992}
\bibinfo{author}{M.~Brune}, \bibinfo{author}{S.~Haroche},
  \bibinfo{author}{J.-M. Raimond}, \bibinfo{author}{L.~Davidovich},
  \bibinfo{author}{N.~Zagury}, \bibinfo{journal}{Physical Review A}
  \bibinfo{volume}{45} (\bibinfo{year}{1992}) \bibinfo{pages}{5193--5214}.
\bibitem[{Somaraju et~al.(2011)Somaraju, Mirrahimi, and Rouchon}]{Somaraju2011}
\bibinfo{author}{R.~Somaraju}, \bibinfo{author}{M.~Mirrahimi},
  \bibinfo{author}{P.~Rouchon}, \bibinfo{journal}{arXiv:1103.1724v1}
  (\bibinfo{year}{2011}).
\bibitem[{Beauchard and Mirrahimi(2009)}]{Beauchard2009}
\bibinfo{author}{K.~Beauchard}, \bibinfo{author}{M.~Mirrahimi},
  \bibinfo{journal}{SIAM J. Contr. Optim.} \bibinfo{volume}{48}
  (\bibinfo{year}{2009}) \bibinfo{pages}{1179--1205}.
\bibitem[{Mirrahimi(2009)}]{Mirrahimi2009}
\bibinfo{author}{M.~Mirrahimi}, \bibinfo{journal}{Ann. IHP Nonlinear Analysis}
  \bibinfo{volume}{2} (\bibinfo{year}{2009}) \bibinfo{pages}{1743--1765}.
\bibitem[{{Beauchard} and {Nersesyan}(2010)}]{Beauchard2010}
\bibinfo{author}{K.~{Beauchard}}, \bibinfo{author}{V.~{Nersesyan}},
  \bibinfo{journal}{ArXiv e-prints}  (\bibinfo{year}{2010}).
\bibitem[{Coron and d'And\'{r}ea Novel(1998)}]{Coron1998}
\bibinfo{author}{J.-M. Coron}, \bibinfo{author}{B.~d'And\'{r}ea Novel},
  \bibinfo{journal}{IEEE Trans. Automat. Control} \bibinfo{volume}{43}
  (\bibinfo{year}{1998}) \bibinfo{pages}{608--618}.
\bibitem[{Coron et~al.(2007)Coron, d'Andr\'{e}a Novel, and Bastin}]{Coron2007}
\bibinfo{author}{J.-M. Coron}, \bibinfo{author}{B.~d'Andr\'{e}a Novel},
  \bibinfo{author}{G.~Bastin}, \bibinfo{journal}{IEEE Transactions on Automatic
  Control} \bibinfo{volume}{52} (\bibinfo{year}{2007}) \bibinfo{pages}{2--11}.
\bibitem[{Teschl(2009)}]{Teschl2009}
\bibinfo{author}{G.~Teschl}, \bibinfo{title}{Mathematical Methods in Quantum
  Mechanics: With Applications to Schr\"{o}dinger Operators},
  volume~\bibinfo{volume}{99} of \textit{\bibinfo{series}{Graduate Studies in
  Mathematics}}, \bibinfo{publisher}{American Mathematical Society},
  \bibinfo{year}{2009}.
\bibitem[{Merkle(2000)}]{Merkle2000}
\bibinfo{author}{M.~Merkle}, \bibinfo{journal}{Zb. radova Mat. Inst.}
  \bibinfo{volume}{9} (\bibinfo{year}{2000}) \bibinfo{pages}{235--274}.
\bibitem[{Billingsley(1999)}]{Billingsley1999}
\bibinfo{author}{P.~Billingsley}, \bibinfo{title}{Convergence of Probability
  Measures}, \bibinfo{publisher}{John Wiley \& Sons, Inc.},
  \bibinfo{year}{1999}.
\bibitem[{Reed and Simon(1980)}]{Reed1980}
\bibinfo{author}{M.~Reed}, \bibinfo{author}{B.~Simon}, \bibinfo{title}{Methods
  of modern mathematical physics}, volume \bibinfo{volume}{1. Functional
  Analysis}, \bibinfo{publisher}{Academic Press}, \bibinfo{year}{1980}.
\bibitem[{Nielsen and Chuang(2000)}]{Nielsen2000}
\bibinfo{author}{M.~A. Nielsen}, \bibinfo{author}{I.~L. Chuang},
  \bibinfo{title}{Quantum Computation and Quantum Information},
  \bibinfo{publisher}{Cambridge University Press}, \bibinfo{year}{2000}.
\bibitem[{Renardy and Rogers(2004)}]{Renardy2004}
\bibinfo{author}{M.~Renardy}, \bibinfo{author}{R.~C. Rogers},
  \bibinfo{title}{An Introduction to Partial Differential Equations},
  \bibinfo{publisher}{Springer-Verlag}, \bibinfo{edition}{2} edition,
  \bibinfo{year}{2004}.
\bibitem[{Kato(1966)}]{Kato1966}
\bibinfo{author}{Kato}, \bibinfo{title}{Perturbation theory for linear
  operators}, \bibinfo{publisher}{Springer}, \bibinfo{year}{1966}.

\end{thebibliography}
\end{document}